\documentclass[a4paper,10pt]{amsart}

\usepackage{amssymb}
\usepackage{amsmath, amscd}
\usepackage{amsthm}
\usepackage{graphicx}
\usepackage{enumerate}

\makeatletter
\newtheorem*{rep@theorem}{\rep@title}
\newcommand{\newreptheorem}[2]{%
\newenvironment{rep#1}[1]{%
 \def\rep@title{#2 \ref{##1}}%
 \begin{rep@theorem}}%
 {\end{rep@theorem}}}
\makeatother

\theoremstyle{plain}
\newtheorem{theorem}{Theorem}[section]
\newreptheorem{theorem}{Theorem}

\newreptheorem{thm}{(A special case of) Theorem}
\newtheorem{lemma}[theorem]{Lemma}
\newtheorem{proposition}[theorem]{Proposition}
\newreptheorem{proposition}{Proposition}
\newtheorem{cor}[theorem]{Corollary}
\newreptheorem{cor}{Corollary}

\theoremstyle{definition}
\newtheorem{definition}[theorem]{Definition}
\newtheorem{example}[theorem]{Example}

\newtheorem{qu}{Question}

\theoremstyle{remark}
\newtheorem{rem}[theorem]{Remark}

\DeclareMathOperator{\aut}{Aut}
\DeclareMathOperator{\out}{Out}
\def\outn{{\rm{Out}}(F_n)}

\def\auto{{\rm{Aut^0}}(A_\G)}
\def\saut{{\rm{PTAut^0}}(A_\G)}
\def\outo{{\rm{Out}}^0(A_\G)}
\def\sout{{\rm{PTOut}}^0(A_\G)}
\def\autg{{\rm{Aut}}(A_\G)}
\def\outg{{\rm{Out}}(A_\G)}

\def\torg{\mathcal{T}(A_\G)}
\def\toro{\overline{\mathcal{T}}(A_\G)}
\def\symg{{\rm{Sym}}(A_\G)}

\DeclareMathOperator{\gl}{GL}
\def\gln{{\rm{GL}}_n(\mathbb{Z})}
\def\slm{{\rm{SL}}_m(\mathbb{Z})}
\def\sln{{\rm{SL}}_n(\mathbb{Z})}
\DeclareMathOperator{\SL}{SL}
\DeclareMathOperator{\im}{im}

\def\r{\rho}

\def\G{\Gamma}
\def\g{\gamma}
\def\Z{\mathbb{Z}}
\def\ssm{\smallsetminus}
\def\env{U}
\def\enva{U^{\infty}}
\def\unia{U^*}
\def\A{{A_\G}}
\def\hom{{\rm{Hom}}}
\def\co{\colon\thinspace}
\def\rrank{{\rm{rank}}_\mathbb{R}}
\def\dsl{d_{SL}}
\def\ad{{\rm{ad}}}

\title[Johnson homomorphisms and actions of lattices on RAAGs]{Johnson homomorphisms and actions of higher-rank lattices on right-angled Artin groups}
\date{12 June 2013}
\author{Richard D. Wade}   
\begin{document}
\begin{abstract}
Let $G$ be a real semisimple Lie group with no compact factors and finite centre, and let $\Lambda$ be an irreducible lattice in $G$. Suppose that there exists a homomorphism from $\Lambda$ to the outer automorphism group of a right-angled Artin group $A_\G$ with infinite image. We give a strict upper bound to the real rank of $G$ that is determined by the structure of cliques in $\G$. An essential tool is the Andreadakis--Johnson filtration of the Torelli subgroup $\torg$ of $\autg$. We answer a question of Day relating to the abelianisation of $\torg$, and show that $\torg$ and its image in $\outg$ are residually torsion-free nilpotent. \end{abstract}

\address{Department of Mathematics, University of Utah, Salt Lake City, UT 84112}
\email{wade@math.utah.edu}

\maketitle

\section{Introduction} 
Right-angled Artin groups have delivered considerable applications to geometric group theory (two examples are the association with special cube complexes  \cite{HW} and Bestvina--Brady groups \cite{BB}). Another recent area of progress has been the extension of rigidity properties held by irreducible lattices in semisimple Lie groups to involve mapping class groups and automorphism groups of free groups \cite{FM,bridson-farb,BW2010}. In this paper we look to what extent this phenomenon extends to the automorphism group of a right-angled Artin group $A_\G$, where $\G$ is the defining graph. If $\G$ is discrete then $A_\G$ is the free group $F_n$, and if $\G$ is complete then $A_\G$ is the free abelian group $\mathbb{Z}^n$. One therefore expects traits shared by both $\mathbb{Z}^n$ and $F_n$ to be shared by an arbitrary RAAG. Similarly, one optimistically hopes that properties shared by both $\gln$ and $\outn$ will also by shared by $\outg$ for an arbitrary right-angled Artin group. For instance, there is a 
Nielsen-type generating set of $\outg$ given by the work of Laurence \cite{MR1356145} and Servatius \cite{MR1023285}, and $\outg$ has finite virtual cohomological dimension \cite{CV2009}.  $\outg$ is residually finite  \cite{CV2010, Minasyan2009}, and for a large class of graphs, $\outg$ satisfies the Tits alternative \cite{CV2010}.

Bridson and the author recently showed that any homomorphism from an irreducible lattice $\Lambda$ in a higher-rank semisimple Lie group to $\outn$ has finite image \cite{BW2010}. A direct translation of this result cannot hold for an arbitrary RAAG, as $\mathbb{Z}^n$ is a RAAG and $\out(\mathbb{Z}^n)=\gln$. However, Margulis' superrigidity implies that if $\Lambda \to \gln$ is a map with infinite image, then the real rank of the Lie group containing $\Lambda$ must be less than or equal to $n-1$, the real rank of $\SL_n(\mathbb{R})$. The aim of this paper is to show that we may effectively bound the rank of an irreducible lattice acting on $A_\G$, and that this bound is determined by obvious copies of $\slm$ in $\outg$. 
We will describe these copies now. Suppose that $\G'$ is a subgraph of $\G$ which is a clique (i.e. any two vertices of $\G'$ are connected by an edge).  A free abelian subgroup of rank $m$ in $A_\G$ does not imply there exists a copy of $\slm$ in $\outg$, however if every vertex of $\G'$ has the same star  then the natural embedding  $\mathbb{Z}^{|V(\G')|} \to A_\G$ induces an injection $\SL_{|V(\G')|}(\mathbb{Z})\to \outg$. Define the $\SL$--dimension of $\G$, written $d_{SL}(\G)$, to be the number of vertices in the largest such subgraph $\G'$. 

\begin{reptheorem}{t:lr}
Let $G$ be a real semisimple Lie group with finite centre, no compact factors, and $\rrank G \geq 2$. Let $\Lambda$ be an irreducible lattice in $G$. If $\rrank G \geq d_{SL}(\G)$, then every homomorphism $f:\Lambda \to \outg$ has finite image.
\end{reptheorem}

A motivating example is the irreducible lattice $\SL_n(\mathbb{Z})$ in the real semisimple Lie group $\SL_n(\mathbb{R})$ of real rank $\rrank \SL_n(\mathbb{R}) = n-1$. The theory of discrete subgroups of Lie groups is mostly used in a black-box fashion in this paper, so if one wishes this paper can be read with only this example in mind. Witte-Morris \cite{W-M} has written a geometrically flavoured introduction to lattices where the reader may find definitions for the terms in Theorem~\ref{t:lr}. 

Our previous observation that $\outg$ contains a copy of $SL_m(\mathbb{Z})$ for $m=\dsl(\G)$ tells us that the bound on $\rrank G$ given in Theorem \ref{t:lr} is the best that one can provide. The above theorem is deduced from the previously mentioned results of Bridson and the author for $\outn$ and of Margulis for $\sln$, combined with the following general algebraic criteria:

\begin{repthm}{t:main}
Suppose that $\dsl (\G) = m$. Let $$F(\G')=\max\{|V(\G')|:\G' \subset \G \text{ and $A_{\G'}\leq A_\G$ is a free group}\}.$$ Let $\Lambda$ be a group. Suppose that for each finite index subgroup $\Lambda'\leq\Lambda$, we have: \begin{itemize}
                                                                                                                                           \item Every homomorphism $\Lambda' \to \slm$ has finite image, 
\item For all $N \leq F(\G)$, every homomorphism $\Lambda' \to \out(F_N)$ has finite image.
                                                                                                                                          \end{itemize}
Then every homomorphism $f:\Lambda \to \outg$ has finite image.
\end{repthm}

The more general version of Theorem~\ref{t:main} in the body of the paper concerns homomorphisms to subgroups of $\outg$ generated by subsets of a natural generating set.

In the free group case, the rigidity result for lattices was deduced from a more general result that showed that $\mathbb{Z}$--averse groups, that is, groups for which no finite index subgroup contains a normal subgroup mapping onto $\mathbb{Z}$, have no interesting maps to $\outn$.  To do this, one requires some deep geometric results on the behaviour of \emph{fully irreducible} elements of $\outn$ \cite{HM,DGO,BF}, and algebraic results about the structure of the Torelli subgroup of $\outn$ (the subgroup consisting of automorphisms that act trivially on $H_1(F_n)$).  The results for $\outg$ in \cite{CV2009,CV2010} come from looking at \emph{projection homomorphisms} which, when $\G$ is connected, allow us to understand $\outg$ in terms of automorphisms of smaller RAAGs. This allows for inductive arguments. Our approach is to combine the projection homomorphisms alluded to above with algebraic results concerning the structure of the Torelli subgroup of $\outg$.  After the background material in the 
following section, the paper is organised like so: Section~\ref{s:lie theory} uses Lie methods pioneered by Magnus \cite{MKS} to study the lower central series of $A_\G$; in particular we study the consecutive quotients $L_c=\g_c(A_\G)/\g_{c+1}(A_\G)$, and the Lie $\mathbb{Z}$--algebra $L$ formed by taking the direct sum $L=\oplus_{c=1}^{\infty}L_c$, where the Lie bracket is induced by taking commutators in $A_\G$. Let $Z(\_)$ denote the centre of a group or Lie algebra, and $pL$ be the ideal of $L$ consisting of elements which are $p$th multiples in $L$. The main technical result concerning $L$ is:

\begin{reptheorem}{c:centres}
If $Z(A_\G)=\{1\}$ then $Z(L)=Z(L/pL)=0$ and $Z(L/(\oplus_{i=c+1}^\infty L_i))=L_c/(\oplus_{i=c+1}^\infty L_i)$.
\end{reptheorem}

We obtain this by looking at the enveloping algebra of $L$, which we call $U$. This has a particularly nice description as a free $\mathbb{Z}$--module with a basis consisting of positive elements of $A_\G$.  This allows us to reduce questions about commutation in $L$ to questions about commutation in $A_\G$.

Given a group $G$, let $\mathcal{T}(G)$ be the \emph{Torelli subgroup} of $\aut(G)$ -- the subgroup of $\aut(G)$ consisting of automorphisms that act trivially on $H_1(G)$ (the name here comes from the analogous situation in mapping class groups). In Section \ref{s:johnson} we describe a central filtration $\torg=G_1 \trianglerighteq G_2 \trianglerighteq G_3 \ldots$ of $\torg$ analogous to the Andreadakis--Johnson filtration of $\mathcal{T}(F_n)$. We have Johnson homomorphisms $$\tau_c\co G_c \to \hom(L_1,L_{c+1})$$ for which $\ker\tau_c=G_{c+1}$. Day \cite{Day09} has shown that $\torg$ is finitely generated by a set $\mathcal{M}_{\G}$, which we describe in Section~\ref{s:torelli intro}. We show that $\tau_1$ maps the elements of $\mathcal{M}_\G$ to linearly independent elements in $\hom(L_1,L_2)$, therefore $H_1(\torg)$ is a free abelian group with a basis given by the image of $\mathcal{M}_\G$. This answers Question~5.4 of \cite{Day09}. In particular, $\mathcal{M}_\G$ is a minimal generating set of $\torg$.
 The above filtration is separating, so that $\cap_{c=1}^{\infty}G_c=\{1\}$, and the groups $L_c$ are free abelian, so each consecutive quotient $G_c/G_{c+1}$ is free abelian. 

Let $\overline{\mathcal{T}}(A_\G)$ be the image of $\mathcal{T}(A_\G)$ in $\out(A_\G)$. In the final part of Section \ref{s:johnson} we  combine Theorem \ref{c:centres} with results of Minasyan \cite{Minasyan2009} and Toinet \cite{Toinet10} to show that the image $H_1,H_2,H_3,\ldots$ of the series $G_1,G_2,G_3, \ldots$ in $\outg$ is separating and that each consecutive quotient $H_c/H_{c+1}$ is also free abelian. (This roughly follows methods designed by Bass and Lubotzky \cite{BL} for studying central series.) In particular, we attain the following result:

\begin{reptheorem}{t:tfn}
For any graph $\G$, the group $\overline{\mathcal{T}}(A_\G)$ is residually torsion-free nilpotent.
\end{reptheorem}

This was also discovered independently by Toinet \cite{Toinet10}. This is a key part in the proof of Theorem~\ref{t:main}. In particular we use it to deal with the inductive step when the underlying graph $\G$ is disconnected. 

A problem arises that projection homomorphisms are not generally surjective; consequently they may raise $\SL$--dimension. This is confronted in Section \ref{s:sldim}, where we study the image of $\outg$ under a projection map. As such subgroups appear naturally when working with projections, we feel that the methods here may be of independent interest. In particular, we define a notion of $\SL$--dimension for an arbitrary subgroup $G \leq \outg$, and show that if we only study the image of $\outg$ under a projection map, $\SL$--dimension does not increase. This prepares us to complete the proof of Theorem \ref{t:main}. We describe applications in Section \ref{s:consequences}.

The author would like to thank his supervisor Martin Bridson for his encouragement and advice, the referee for feedback which has considerably improved this paper, and Andrew Sale for a series of enthusiastic and helpful conversations.

\section{Background} \label{s:background1}

Most of this section is standard, however there are some ideas here that are new. In Section \ref{s:generators} we extend the usual generating set of $\outg$ to include \emph{extended partial conjugations}, which are products of partial conjugations that conjugate by the same element of $A_\G$. We also order the vertices of $\G$ in a useful way in Sections \ref{s:relation} and \ref{s:gordering}, which will be used later in the paper to give a block decomposition of the action of subgroups of $\outg$ on $H_1(A_\G)$. 

We should first define $A_\G$. Let $\G$ be a graph with vertex and edge sets $E(\G)$ and $V(\G)$ respectively.  If $\iota, \tau:E(\G) \to V(\G)$ are the maps that take an edge to its initial and terminal vertices, then $A_\G$ has the presentation:$$ A_\G= \langle v:v \in V(\G)|[\iota(e),\tau(e)]:e \in E(\G) \rangle.   $$ Our commutator convention is $[g,h]=ghg^{-1}h^{-1}$. Throughout we assume that $\G$ has $n$ vertices labelled $\{v_1,\ldots,v_n\}$. Given a vertex $v$ in $\G$, the \emph{link of} $v$, or $lk(v)$, is the set of vertices of $\G$ adjacent to $v$. The \emph{star} of $v$, or $st(v)$, is defined by $st(v)=lk(v) \cup \{v\}$. 

 \subsection{Words in $A_\G$.}\label{s:words}
At times we will need to discuss words and word length in $A_\G$. The support, $supp(w)$, of a word $w$ on $V(\G)\cup V(\G)^{-1}$ is the full subgraph of $\G$ spanned by the generators (or their inverses) occurring in $w$. We say that a word $w$ representing $g \in A_\G$ is \emph{reduced} if there is no sub-word of the form $v^{\pm1} w'v^{\mp1}$ with $supp(w')\subset st(v)$. We may pass between two reduced words by repeatedly switching consecutive letters that commute (see \cite{MR2322545}) -- it follows that we may define $supp(g)$ to be the support of any reduced word representing $g$. Similarly the \emph{length} of $g$ is the length of any reduced word representing $g$. 

\subsection{A generating set of $\autg$} \label{s:generators} Laurence \cite{MR1356145} proved a conjecture of Servatius \cite{MR1023285} that $\autg$ has a finite generating set consisting of the following automorphisms:

\begin{description} 
\item[Graph symmetries]
If a permutation of the vertices comes from an isomorphism of the graph, then this permutation induces an automorphism of $A_\G$. These automorphisms form a finite subgroup of $\autg$ called $\symg$.
\item[Inversions] These are automorphisms that come from inverting one of the generators of $A_\G$, so that:
$$s_i(v_k)=\begin{cases}  v_i^{-1}  & i=k \\
v_k  & i \neq k. \end{cases}$$
\item[Partial conjugations]
Suppose $[v_i,v_j] \neq 0$. Let $\G_{ij}$ be the connected component of $\G - st(v_j)$ containing $v_i.$ Then the partial conjugation $K_{ij}$ conjugates every vertex of $\G_{ij}$ by $v_j$, and fixes the remaining vertices, so that:
$$K_{ij}(v_k)=\begin{cases} v_jv_kv_j^{-1} & v_k \in \G_{ij} \\
       v_k & v_k \not \in \G_{ij}. \end{cases} $$
Note that if $lk(v_i) \subset st(v_j)$ then $\G_{ij}=\{v_i\}$, so in this case $K_{ij}$ fixes every basis element except $v_i$.
\item[Transvections]
If $lk(v_i) \subset st(v_j)$, then there is an automorphism $\r_{ij}$ which acts on the generators of $A_\G$ as follows:
$$\r_{ij}(v_k)=\begin{cases} v_iv_j  & i=k \\
v_k  & i \neq k. \end{cases} $$
\end{description}

There are two important finite index normal subgroups of $\autg$ that we obtain from this classification. The first is the subgroup generated by inversions, partial conjugations, and transvections and is denoted $\auto$. The second is the smaller subgroup generated by only partial conjugations and transvections. Denote this group $\saut$. In some cases we will need to look at groups generated by (outer) automorphisms that conjugate more than one component of $\G - st(v_j)$ by $v_j$. Let $T$ be a subset of $\G - st(v_j)$ such that no two vertices of $T$ lie in the same connected component of $\G - st(v_j)$. We define an \emph{extended partial conjugation} to be an automorphism of the form $\prod_{t \in T}K_{tj}$.  We will abuse notation by describing the images of the above elements in $\outg$ by the same names, so that the groups $\outo$ and $\sout$ are defined in the same manner. If $\phi \in \autg$, we use $[\phi]$ to denote the equivalence class of $\phi$ in $\outg$. 

\begin{definition}Let $\mathcal{S}_\G$ be the enlarged generating set of $\outg$ given by graph symmetries, inversions, extended partial conjugations, and transvections. \end{definition}

We shall be studying subgroups of $\outg$ generated by subsets of $\mathcal{S}_{\G}$, however some of these groups are not generated by subsets of the standard generating set. This is because under the restriction, exclusion and projection maps defined in Section \ref{s:rep}, partial conjugations are not always mapped to partial conjugations, but are always mapped to extended partial conjugations. Throughout this paper, we will assume $\autg$ and $\outg$ act on $A_\G$ on the left.

\subsection{Ordering the vertices of $\G$} \label{s:relation}

Extending the definition of the link and star of a vertex, given any full subgraph $\G'$ of $\G$, the subgraph $lk(\G')$ is defined to be the intersection of the links of the vertices of $\G'$, and we define $st(\G')=\G' \cup lk(\G').$  Given any full subgraph $\G'$, the right-angled Artin group $A_{\G'}$ injects into $A_\G,$ so can be viewed as a subgroup. We shall only consider full subgraphs of $\G$, so will sometimes blur the distinction between a subset of the vertices of $\G$ and the full subgraph spanned by these vertices.

\begin{definition}[Standard ordering]
The \emph{standard order} on $V(\G)$ is the binary relation on $V(\G)$ defined by $u \leq v$ if $lk(u) \subset st(v).$ \end{definition} This ordering was first introduced in \cite{KMNR}, where it was shown that this relation is transitive as well as reflexive, so defines a preorder on the vertices. This induces an equivalence relation by letting $u \sim v$ if $u \leq v$ and $v \leq u$. Let $[v]$ denote the equivalence class of the vertex $v.$  We will abuse notation by also using $[v]$ to denote the full subgraph of $\G$ spanned by the vertices in this equivalence class. The preorder descends to a partial order of the equivalence classes. We say that $[v]$ is \emph{maximal} if it is maximal with respect to this ordering. The vertices of $[v]$ are all either at distance one from each other in $\G$, or generate a free subgroup of $A_\G$, therefore $A_{[v]}$ is either a free abelian, or a free non-abelian subgroup of $A_\G$ (\cite{CV2009}, Lemma 2.3). We say that the equivalence class $[v]$ is \emph{abelian} or 
\emph{non-abelian} respectively. Suppose that there are $r$ equivalence classes of vertices in $\G$. We may choose an enumeration of the vertices so that there exists $1=m_1<m_2<\ldots<m_r < n$ such that the equivalence classes are the sets $\{v_1=v_{m_1},\ldots,v_{m_2-1}\},\ldots,\{v_{m_r},\ldots,v_n\}.$ With further rearrangement we may assume that $v_{m_i} \leq v_{m_j}$ only if $i \leq j$. We formally define $m_{r+1}=n+1$ so that for all $i$, the equivalence class of $[v_{m_i}]$ contains $m_{i+1}-m_i$ vertices. 

\subsection{G--ordering vertices}\label{s:gordering}

Let $G$ be a subgroup of  $\outg$. \begin{definition}[$G$--order] The \emph{$G$--order} on $V(\G)$ is the binary relation on $V(\G)$ defined by $v_i \leq_G v_j$ if either $i=j$ or $[\rho_{ij}] \in G$. \end{definition}

As $\r_{ij}$ is defined if and only if $lk(v_i) \subset st(v_j)$, when $G=\outg$ the $G$--order is simply the standard order defined above. In general, the $G$--order is a subset of the standard order defined on the vertices. Furthermore,  $\leq_G$ is reflexive by definition and transitive as $\r_{il}=\r_{jl}^{-1}\r_{ij}^{-1}\r_{jl}\r_{ij}$. Hence $\leq_G$ is a preorder on $V(\G)$. As with the standard order, $\leq_G$ induces an equivalence relation $\sim_G$ on the vertices, and induces a partial ordering of the equivalence classes of $\sim_G$. Let $[v_i]_G$ be the equivalence class of the vertex $v_i$. As $\leq_G$ is a subset of $\leq$, each equivalence class $[v_i]_G$ is a subset of the equivalence class $[v_i]$. In particular the subgroup $A_{[v_i]_G}$ is either free abelian or free and non-abelian, so $[v_i]_G$ may also be described as $\emph{abelian}$ or $\emph{non-abelian}$.  Suppose that there are $r' \geq r$ equivalence classes of vertices in $\sim_G$. We may further refine the enumeration of the 
vertices given previously so that there exists $1=l_1 < l_2 < \ldots < l_{r'} < n$ such that the equivalence classes of $\sim_G$ are the sets $\{v_1=v_{l_1},\ldots,v_{l_2-1}\},\ldots,\{v_{l_{r'}},\ldots,v_n\},$ and $v_{l_i} \leq_G v_{l_j}$ only if $i \leq j$. Define $l_{r'+1}=n+1$ so that for all $i$, the equivalence class of $[v_{l_i}]_G$ contains $l_{i+1}-l_{i}$ vertices.

\subsection{The Torelli subgroups of $\autg$ and $\outg$}\label{s:torelli intro}

The abelianisation of $A_{\G}$, given by $H_1(A_\G)=A_\G/[A_\G,A_\G]$, is a free abelian group generated by the images of the vertices under the abelianisation map $A_\G \to A_\G/[A_\G,A_\G].$ This induces a natural map $$\Phi:\autg \to \aut(H_1(A_\G)).$$ Once and for all we fix the basis of $H_1(A_\G)$ to be the image of $V(\G)$ under the abelianisation map. This allows us to identify $\aut(H_1(A_\G))$ with $\gln$. This is the viewpoint which we will take for the rest of the paper. We say that $\ker\Phi=\torg$ is the \emph{Torelli subgroup of} $\autg$. Every partial conjugation lies in $\torg$. If $v_i,v_j$ and $v_k$ are distinct vertices, with $lk(v_i) \subset st(v_j)\cap st(v_k)$ and $[v_j,v_k]\neq0,$ then the mapping $$K_{ijk}(v_l)=\begin{cases} v_i[v_j,v_k] & l=i \\
v_l & l \neq i, \end{cases} $$ defines a nontrivial element of $\torg$. 
       
\begin{definition}\label{d:mg}
Let $\mathcal{M}_\G$ be the subset of $\autg$ consisting of: \begin{enumerate} \item Partial conjugations. \item Elements of the form $K_{ijk}$, where $[v_j,v_k]\neq1$, $lk(v_i) \subset st(v_j)\cap st(v_k)$, and $j<k$. \end{enumerate} \end{definition}

We add the restriction that $j <k$ in the above definition as  $K_{ijk}=K_{ikj}^{-1}$. In \cite{Day09}, Day proves the following theorem:\begin{theorem}[\cite{Day09}, Theorem B] $\mathcal{M}_\G$ is a generating set of $\torg$. \end{theorem}
       
Let $\text{Inn}(A_\G)$ denote the set of inner automorphisms of $A_\G$. As $\Phi$ sends every element of $\text{Inn}(A_\G)$ to the identity, we can factor out $\text{Inn}(A_\G)$ to obtain a map $$\overline{\Phi}:\outg \to \gln.$$
We define $\toro=\ker\overline{\Phi},$ and call this group the \emph{Torelli subgroup of} $\outg$.

\subsection{Restriction, exclusion, and projection homomorphisms.} \label{s:rep}

Suppose that $\G'$ is a full subgraph of $\G$, so that $A_{\G'}$ can be viewed as a subgroup of $A_\G$ in the natural way. We say that \emph{the conjugacy class of $A_{\G'}$ is preserved by} $G < \out(A_\G)$ if for every element $[\phi] \in G$ there exists a representative $\psi \in [\phi]$ such that $\psi(A_{\G'})=A_{\G'}$.

\begin{proposition}If $G$ preserves the conjugacy class of $A_{\G'}$ then there is a homomorphism $R_{\G'}: G \to \out(A_{\G'})$. \end{proposition}

\begin{proof}
Let $[\phi] \in G$. Let $\psi$ and $\psi'$ be two representatives of $[\phi]$ such that $\psi|_{A_{\G'}}$ and $\psi'|_{A_{\G'}}$ are automorphisms of $A_{\G'}$. Then $\psi$ and $\psi'$ differ by an inner automorphism $\ad_g$ with $gA_{\G'}g^{-1}=A_{\G'}$.  By Proposition 2.2 of  \cite{MR2372847} there exists $g_1 \in A_{\G'}$ and $g_2$ in the centraliser of $A_{\G'}$ such that $g=g_1g_2$. It follows that $\psi|_{A_{\G'}}$ and $\psi|_{A_{\G'}}$ differ by $\ad_{g_1}$ and represent the same elements of $\out(A_{\G'})$. Hence $[\phi] \mapsto [\psi|_{A_{\G'}}]$ gives the required homomorphism.\end{proof}

We say that $R_{\G'}$ is a \emph{restriction map}. Rather than asking that $G$ preserves the conjugacy class of $A_{\G'}$ we may impose the weaker requirement that the normal closure $\langle \langle A_{\G'}\rangle \rangle$ of $A_{\G'}$ is fixed setwise by a (equivalently, any) representative of each element of $G$. In this case, we say that $G$ \emph{preserves the normal closure of} $A_{\G'}$, and one can show the following: 

\begin{proposition}
Suppose that the normal closure of $A_{\G'}$ is preserved by a subgroup $G< \outg$. Then there is a homomorphism $$E_{\G'}\co G \to \out(A_\G/\langle \langle A_{\G'}\rangle \rangle)\cong\out(A_{\G-\G'}),$$ which we call the exclusion map.  \end{proposition}

\begin{example}\label{e:maps}
If $\G$ is connected and $v$ is a maximal vertex with respect to the standard ordering then the conjugacy classes of $A_{[v]}$  and $A_{st[v]}$ are preserved by $\outo$ (\cite{CV2009}, Proposition 3.2). Therefore there is a restriction map $$R_v:\outo \to \out^0(A_{st[v]}),$$ an exclusion map $$E_v:\outo \to \out^0(A_{\G - [v]}),$$ and a projection map $$P_v:\outo \to \out^0(A_{lk[v]})$$ obtained by combining the restriction and exclusion maps. We can take the direct sum of these projection maps over all maximal equivalence classes $[v]$ to obtain the \emph{amalgamated projection homomorphism}: $$P:\outo \to \bigoplus_{\text{$[v]$ maximal}}\out^0(A_{lk[v]})$$  
\end{example}

\begin{example}\label{e:disconnected} If $\G$ is not connected, then there exists a finite set $\{\G_i\}_{i=1}^k$ of connected graphs containing at least two vertices and an integer $N$ such that $A_\G \cong F_N \ast_{i=1}^k A_{\G_i} $. Each generator of $\outo$ sends $A_{\G_i}$ to a conjugate of itself, therefore the conjugacy class of $A_{\G_i}$ is preserved by $\outo$, and for each $i$ we obtain a restriction map
$$R_i:\outo \to \out^0(A_{\G_i}).$$
Furthermore, the normal closure of $\ast_{i \in I}A_{\G_i}$ is preserved by $\outg$ and there is an exclusion map $$E: \outg \to \out(F_N).$$
\end{example}

Charney and Vogtmann have shown that when $\G$ is connected the maps in Example \ref{e:maps} describe $\outg$ almost completely. There are two cases: when the centre of $A_\G$, which we write as $Z(A_\G)$, is trivial, and when $Z(A_\G)$ is nontrivial. In the first case, they show the following:

\begin{theorem}[\cite{CV2009}, Theorem 4.2] \label{t:projections}
If $\G$ is connected and $Z(A_\G)$ is trivial, then $\ker P$ is a finitely generated free abelian group.
\end{theorem} When $Z(A_\G)$ is nontrivial there is a unique maximal abelian equivalence class $[v]$. Also $A_{[v]}=Z(A_\G)$ and we are in the following situation:

\begin{proposition}[\cite{CV2009}, Proposition 4.4]\label{p:projections2}
If $Z(A_\G)=A_{[v]}$ is nontrivial, then $$\outg \cong \text{Tr} \rtimes (\text{GL}(A_{[v]}) \times \out(A_{lk[v]})),$$
where $\text{Tr}$ is the free abelian group generated by the transvections $[\r_{ij}]$ such that $v_i \in lk[v]$ and $v_j \in [v]$. The map to $\text{GL}(A_{[v]})$ is given by the restriction map $R_v$, and the map to $\out(A_{lk[v]})$ is given by the projection map $P_v$. The subgroup $\text{Tr}$ is the kernel of the product  map $R_v \times P_v.$
\end{proposition}

In the above proposition we do not need to restrict $R_v$ and $P_v$ to $\outo$, as every automorphism of $A_\G$ preserves $Z(A_\G)=A_{[v]}$. When $\G$ is disconnected, the restriction and exclusion maps of Example \ref{e:disconnected} give us less information. As above, we may amalgamate the restriction maps $R_i$ and the exclusion map $E$, however in this situation the kernel of the amalgamated map is more complicated.
\section{The lower central series of $A_\G$} \label{s:lie theory}

In this section we shall gather some results on the lower central series of $A_\G$ and its associated Lie algebra that we require in the rest of the paper. In Proposition~\ref{l2basis} we give a basis for the free abelian group $\g_2(A_\G)/\g_3(A_\G)$ and in Theorem~\ref{c:centres} we give information about the structure of the Lie algebra $L=\oplus_{i=1}^\infty\g_i(A_\G)/\g_{i+1}(A_\G).$ From Magnus' work, most of the results are well known in the free group case (see, for example, Chapter 5 of \cite{MKS} or Chapter 2 of \cite{MR979493}). The extension to right-angled Artin groups originally appeared in Droms' thesis \cite{Droms2} and in papers of Duchamp and Krob \cite{MR1176154,DK1}. Right-angled Artin groups are known as \emph{graph groups} and \emph{partially commutative groups} respectively in these works. See \cite{W1} for a survey of some of the key results in this area.

\subsection{The Lie algebra $L$ and its enveloping algebra $U$} \label{s:background2}

Let $\g_c(A_\G)$ be the $c$th term in the lower central series of $A_\G$, so that $\g_1(A_\G)=A_\G$ and $\g_{c+1}(A_\G)=[\g_c(A_\G),A_\G]$. As we are keeping $A_\G$ fixed throughout we shall often shorten $\g_c(A_\G)$ to simply $\g_c$. Let $L_c=\g_c/\g_{c+1}$. Let $L= \oplus_{i=1}^\infty L_i$. An element of $L$ is of the form $\sum_{i=1}^\infty g_i.\g_{i+1}$ with $g_i \in \g_i$, and all but finitely many $g_i$ are equal to the identity. As $[\g_c,\g_d]\subset\g_{c+d}$, the $\mathbb{Z}$--module $L$ inherits a graded $\mathbb{Z}$--Lie algebra structure by taking commutators in $A_\G$ as follows: $$[\sum_{i=1}^\infty g_i.\g_{i+1}, \sum_{i=1}^\infty h_i.\g_{i+1}]=\sum_{c=2}^\infty \sum_{i+j=c}[g_i,h_j].\g_{c+1}.$$ $L$ is generated by the images of $v_1,\ldots,v_n$ in $L_1$. 

\begin{definition}
We say that an element $g \in A_\G$ is \emph{positive} if it is represented by a positive word in $V(\G)$ (without using letters of the form $v^{-1}$). 
\end{definition}

The set of positive elements in $A_\G$ forms a monoid $M$. Let $M_i$ be the subset of $M$ consisting of positive words of length $i$. Then $M_0$ is the set containing the identity element (the only word of length 0).

\begin{definition}
Let $U$ be the free $\Z$--module with a basis given by elements of $M.$ Let $U_i$ be the submodule of $U$ spanned by $M_i$. Then $U=\oplus_{i=0}^{\infty}U_i$, and multiplication in $A_\G$ gives $U$ the structure of a graded $\mathbb{Z}$--algebra.
\end{definition}

We will distinguish elements of $\env$ from $A_\G$ by writing positive elements in $\{\mathbf{v_1,\ldots,v_n}\}$ rather than $\{v_1,\ldots,v_n\}$.  Let $\enva$ be the algebra extending $\env$ by allowing infinitely many coefficients of a sequence of positive elements to be non-zero. Any element of $\enva$ can be written uniquely as a power series $a=\sum_{i=0}^{\infty}a_i,$ where $a_i$ is an element of $U_i.$ We say that $a_i$ is the \emph{homogeneous part} of $a$ of degree $i$. Each $a_i$ is a linear sum of positive elements of length $i$, so is of the form $a_i=\sum_{k=1}^m\lambda_k g_k$, where $g_k$ is a positive element in $A_\G$ of length $i$ and $\lambda_k \in \mathbb{Z} \setminus \{0\}.$ Let $\unia$ be the group of units in $\enva$. If $a$ is of the form $a=1 + \sum_{i=1}^{\infty}a_i,$ then $a \in \unia$ and $$a^{-1}=1-(a_1+a_2+\cdots)+(a_1+a_2+\cdots)^2-\ldots=1 + \sum_{i=1}^{\infty}b_i,$$ where $b_1=-a_1$ and $b_j=-a_j-\sum_{i=1}^{j-1}b_ia_{j-i}$ recursively. We have abused notation slightly in the 
above by writing $1$ as the leading coefficient rather than $1.1_{A_\G}.$ We may use $\enva$ to study $A_\G$ via the following theorem:

\begin{theorem}[{Droms \cite{Droms2}, Duchamp and Krob \cite[Theorem 1.2]{DK1}}]The mapping $v_i \mapsto 1 + \mathbf{v_i}$ induces an injective homomorphism $\mu:A_\G \to \unia.$ \end{theorem}

$\mu$ is sometimes called the \emph{Magnus map} or \emph{Magnus morphism}. There is a central series $D_1 \geq D_2 \geq D_3 \ldots$ of $\unia$ defined by letting $a \in D_c$ if and only if $a_i=0$ when $0 < i < c.$ It is related to the lower central series of $A_\G$ by the following theorem: 

\begin{theorem}[\cite{DK1}, Theorem 2.2] \label{mu is nice}
For all $c$ we have $\mu^{-1}(D_c)=\g_c.$ 
\end{theorem}

As $\cap_{c=1}^{\infty}D_c=\{\pm 1\}$, it follows that $\cap_{c=1}^{\infty} \g_c= \{1\},$ hence $A_\G$ is residually nilpotent. This gives one an effective way of studying the lower central series of $\A.$ 

Let $\mu(g)_i$ be the $i$th homogeneous part of $\mu(g)$. Let $\mathcal{L}(\env)$ be $\env$ endowed with the Lie bracket defined by $[u_1,u_2]=u_1.u_2-u_2.u_1.$ Then $U$ and $L$ are related by the following theorem:

\begin{theorem}\label{t:env}$L$ is a free $\mathbb{Z}$--module, and therefore each $L_i=\g_i/\g_{i+1}$ is a free abelian group. There is a homomorphism $\alpha_i:L_i \to U_i$ given by $g.\g_{i+1} \mapsto \mu(g)_i$ which induces an injective Lie algebra homomorphism $\alpha \co L \to \mathcal{L}(U),$ so that $\alpha(L)$ is the Lie subalgebra of $\mathcal{L}(U)$ generated by $\{\mathbf{v_1,\ldots,v_n}\}$.
\end{theorem}

\begin{proof}Duchamp and Krob showed that $L$ is isomorphic to a \emph{free partially commutative Lie $\mathbb{Z}$--algebra} \cite[Theorem 2.1]{DK1} and that this algebra is free as a $\mathbb{Z}$--module \cite[Corollary II.16]{MR1176154}. They also show that $\mathcal{L}(U)$ is its enveloping algebra \cite[Corollary I.2]{MR1176154}, and as $L$ is free as a $\mathbb{Z}$--module, the natural map from $L$ into its enveloping algebra is injective by the Poincar\'e--Birkhoff--Witt Theorem \cite[I.2.7 Corollary 2]{MR979493}. It only remains to check that $\alpha$ is this natural map. This is described in Section~6 of \cite{W1}. \end{proof}

It is helpful to get get a clear picture of a basis for $L$ (or equivalently, of $\alpha(L)$). We use the basis consisting of \emph{Lyndon heaps}, or \emph{Lyndon elements}, introduced by Lalonde \cite{L1}. We do not need a complete description of a basis, but will make use of the following theorem:

\begin{theorem}[Lalonde, \cite{L1}] \label{t:basis} There is a total ordering $<_M$ of $M$ and a basis $B=\{\beta(g)\}_{g \in LH}$ of $L$ indexed by a subset $LH \subset M$ such that: \begin{enumerate} \item Each $\alpha(\beta(g)) \in U$ is a homogeneous element of degree $|g|$. \item The coefficient of $g$ in $\alpha(\beta(g)) \in U_{|g|}$ is equal to 1. \item If the coefficient of $h$ is non-zero in $\alpha(\beta(g))$ then $g\leq_M h$.\end{enumerate} \end{theorem}

The element $\beta(g)\in L$ is obtained from $g$ by an appropriate bracketing operation. Theorem~\ref{t:basis} has the following useful corollary:

\begin{cor}\label{c:coeff}
Let $pL$ and $pU$ be the ideals consisting of $p$th multiples of elements in $L$ and $U$ respectively. Let $x \in L$. Then $x \in pL$ if and only if $\alpha(x) \in pU$.
\end{cor}

\begin{proof}                                                                                                                                                                                                                                                                                                                                                                                                                                                                                                                                                                                                                                                                                                                                                                    
As $\alpha$ is a homomorphism, if $x \in pL$ then $\alpha(x) \in pU$. Conversely, suppose that $\alpha(x) \in pU$, so that the coefficient of every term in $\alpha(x)$ is divisible by $p$. By Theorem~\ref{t:basis}, there exist $g_1,\ldots,g_m \in LH$ and $i_1,\ldots,i_m \in \mathbb{Z}$ such that $$\alpha(x)=i_1\alpha(\beta(g_1))+i_2\alpha(\beta(g_2))+\cdots + i_m\alpha(\beta(g_m)).$$ We may also assume that $g_1 <_M g_2 <_M \cdots <_M g_m$. By Parts 2 and 3 of Theorem~\ref{t:basis}, the element $g_1$ is minimal in the decomposition of $\alpha(x)$ and has coefficient $i_1$. Hence $i_1=pj_1$ for some $j_1 \in \mathbb{Z}$. Then $\alpha(x - pj_1\beta(g_1)) \in pU$ and $$\alpha(x - pj_1\beta(g_1))= i_2\alpha(\beta(g_2))+\cdots + i_m\alpha(\beta(g_m)).$$ By the same argument as above $p$ divides $i_2$, and continuing by induction we find that $p$ divides $i_k$ for all $k$. Hence $x \in pL$. \end{proof}

\subsection{More information on the structure of $L$}

We now use our associative algebra to give us information about the the structure of $L$.

\begin{proposition} \label{l2basis}
Let $S=\{[v_i,v_j]:i < j, [v_i,v_j]\neq 0\}$. The set $S.\g_3$ is a basis of the free abelian group $L_2=\g_2/\g_3$.
\end{proposition}

\begin{proof} 
The group $\g_2=[A_\G,A_\G]$ is the normal closure of $S$, therefore any element $g \in \g_2$ is of the form $g=g_1s_1^{e_1}g_1^{-1}\cdots g_k s_k^{e_k} g_k^{-1}$, where $s_i \in S$ and $g_i \in A_\G$. However, $[g_i,s_i] \in \g_3$ for all $i$, therefore the image of $g$ in $\g_2/\g_3$ is equal to $s_1^{e_1}\cdots s_k^{e_k}.\g_3$ and $L_2$ is generated by the set $S.\g_3$. We have \begin{align*}\mu([v_i,v_j])&=(1+\mathbf{v_i})(1+\mathbf{v_j})(1+\mathbf{v_j})^{-1}(1+\mathbf{v_i})^{-1} \\ &=(1+\mathbf{v_i})(1+\mathbf{v_j})(1-\mathbf{v_i}+\mathbf{v_i}^2-\cdots)(1-\mathbf{v_j}+\mathbf{v_j}^2-\cdots) \\&=1+\mathbf{v_iv_j-v_jv_i}+ \cdots\end{align*} Under the homomorphism $\alpha_2:L_2 \to U_2$ given in Theorem~\ref{t:env} we have $$\alpha_2([v_i,v_j].\g_3)=\mu([v_i,v_j])_2=\mathbf{v_iv_j-v_jv_i}.$$ The free abelian group (or free $\mathbb{Z}$--module) $U_2$ has a basis consisting of positive elements of length $2$ in $A_\G$, the set: $$\{\mathbf{v_i^2}:v_i \in V(\G)\}\cup\{\mathbf{v_iv_j}:v_i,v_j\in V(\G) \text{ 
and $i <j$ or $[v_i,v_j]\neq0$}\}.$$ The image of $S.\g_3$ under $\alpha_2$ is linearly independent in $U_2$, and therefore forms a basis of $L_2$.
\end{proof}

We shall use Proposition \ref{l2basis} in Section \ref{s:johnson} to describe the abelianisation of $\torg$.  Next, we look at the bracket operation in $U$. Let $Z(\_)$ denote the centre of a group or Lie algebra. 

\begin{proposition} \label{p:centres}
Suppose that $Z(A_\G) = \{1\}$. Let $a=\sum_{k=1}^m\lambda_kg_k$ be a (not necessarily homogeneous) element of $U$. Then for each $k$ there exists $v \in V(\G)$ such that the coefficient of $g_k\mathbf{v}$ in $[a,\mathbf{v}]$ is $\lambda_k$.
\end{proposition}

\begin{proof}
As we can move between any two reduced words representing $g_k$ by a sequence of swaps of consecutive commuting letters the set of vertices that can occur at the start of reduced word representing $g_k$ forms a clique in $\G$ -- all such vertices commute. Call this set $init(g_k)$. Pick $w \in init(g_k)$. As $Z(A_\G)=\{1\}$ there exists $v \in V(\G)$ such that $[v,w]\neq1$. Hence $v \not \in init(g_k)$. We want to look at the coefficient of $g_k\mathbf{v}$ in $$[a,\mathbf{v}]=\sum_{k=1}^m\lambda_kg_k\mathbf{v}-\sum_{k=1}^m\lambda_k\mathbf{v}g_k.$$ Clearly $g_k\mathbf{v}=g_l\mathbf{v}$ if and only if $k=l$. Also, $w \in init(g_k\mathbf{v})$ so $v \not \in init(g_k\mathbf{v})$. As $v \in init(\mathbf{v}g_l)$ for all $l$ we have $g_k\mathbf{v}\neq \mathbf{v}g_l$. Hence the coefficient of $g_k\mathbf{v}$ in $[a,\mathbf{v}]$ is $\lambda_k$.
\end{proof}

The above lets us control the centre of $L$, as well as the centre of its quotient by the ideal $pL$ consisting of $p$th multiples of elements in $L$, and its quotient by the ideal $\oplus_{i=c+1}^\infty L_i$. We first frame this in terms of elements of $A_\G$:

\begin{proposition} \label{pp:centres}
Suppose that $Z(A_\G)=\{1\}$ and $g \in \g_c$. Then \begin{itemize} \item If $[g,v] \in \g_{c+2}$ for all $v \in V(\G)$ then $g \in \g_{c+1}$. \item If for each $v \in V(\G)$ there exists $w \in \g_{c+1}$ such that $[g,v].\g_{c+2}=w^p.\g_{c+2}$ then there exists $h \in \g_{c}$ such that $g.\g_{c+1}=h^p.\g_{c+1}$. \end{itemize}
\end{proposition}

\begin{proof}
Suppose that $g \not \in \g_{c+1}$, and let $\alpha(g.\g_{c+1})=\sum_{i=1}^k\lambda_kg_k$ with $\lambda_k \neq 0$ for all $k$. For every $v \in V(\G)$, we have: \begin{align*}\alpha([g,v].\g_{c+2})&=\alpha([g.\g_{c+1},v.\g_2]) \\ &=[\alpha(g.\g_{c+1}),\alpha(v.\g_2)]\\ &= [\alpha(g.\g_{c+1}), \mathbf{v}]. \end{align*} By Proposition~\ref{p:centres}, for all $k$ there exists $v \in V(\G)$ such that the coefficient of $g_k\mathbf{v}$ in $\alpha([g,v].\g_{c+2})$ is equal to $\lambda_k$. In particular $[g,v].\g_{c+2} \neq 0$ and $[g,v] \not \in \g_{c+2}.$ This completes the first part of this proposition. For the second part, assume further that $g.\g_{c+1}$ is not a $p$th power in $\g_c/\g_{c+1}$. Then $g.\g_{c+1}$ does not lie in $pL$, and $\alpha(g.\g_{c+1})$ does not lie in $pU$ by Corollary~\ref{c:coeff}. Therefore we can choose $\lambda_k$ and $v$ as above so that $p$ does not divide $\lambda_k$, and so that $\alpha([g,v].\g_{c+2})$ is not a $p$th multiple in $U$. Hence there does not exist $w$ such that $[
g,v].\g_{c+2}=w^p.\g_{c+2}$.
\end{proof}

Given $x \in L \ssm \{0\}$ the above proposition tells us that we can find $v \in V(\G)$ such that $[x,v.\g_2] \neq 0$. If $x \not \in \oplus_{i=c}^\infty L_i$ we can choose $v$ such that $[x,v.\g_2] \not \in \oplus_{i=c+1}^\infty L_i$. Finally, if $x \not \in pL$ we can choose $v$ such that $[x,v.\g_2] \not \in pL$. Hence: 

\begin{theorem} \label{c:centres}
If $Z(A_\G)=\{1\}$ then $Z(L)=Z(L/pL)=0$ and $Z(L/(\oplus_{i=c+1}^\infty L_i))=L_c/(\oplus_{i=c+1}^\infty L_i)$. \end{theorem}

\section{The Andreadakis--Johnson Filtration of $\torg$}\label{s:johnson}

In this section we follow the methods of Bass and Lubotzky \cite{BL} to extend the notion of \emph{higher Johnson homomorphisms} from the free group setting to general right-angled Artin groups. Coupled with the work in the previous sections of the paper, these allow us to describe the abelianisation of $\torg$, and show that $\torg$ has a separating central series $G_1,G_2,G_3,\ldots$ where each quotient $G_i/G_{i+1}$ is a finitely generated free abelian group.  This was first studied in the case of free groups by Andreadakis. We show that the image of this series in $\outg$ satisfies the same results.

\subsection{Definition and application to $H_1(\torg)$}

As $\g_c$ is characteristic, there is a natural map $\autg \to \aut(A_\G/\g_c)$. Let $G_{c-1}$ be the kernel of this map. Then $G_0=\autg$ and $G_1=\torg$. Let $\bar{g}$ denote the image of $g$ in $H_1(A_\G)$. The following proposition is proved in the same way as Proposition 2.2 of \cite{BW2010}.

\begin{proposition}
Let $\phi \in G_c$, where $c \geq 1$. There is a homomorphism $\tau_c(\phi):L_1 \to L_{c+1}$ defined by $\tau_c(\phi)(\bar{g})=\phi(g)g^{-1}.\g_{c+2}$. The map $$ \tau_c\co G_c \to \hom(L_1,L_{c+1}) $$ is also homomorphism with $\ker(\tau_c)=G_{c+1}$.
\end{proposition}

As $\cap_{c=1}^{\infty} \g_c=\{1\}$ (from Theorem~\ref{mu is nice}), it follows that $\cap_{c=1}^{\infty}G_c=\{1\}.$ As $L_1$ and $L_{c+1}$ are free abelian (Theorem~\ref{t:env}), $\hom(L_1,L_{c+1})$ is free abelian, and therefore $G_c/G_{c+1}$ is free abelian. One can check that $[G_c,G_d] \subset G_{c+d}$, so that $G_1,G_2,G_3,\ldots$ is a central series of $\torg$. The rank of each $L_c$ has been calculated in \cite{DK1}, although more work is needed to calculate the ranks of the quotients $G_c/G_{c+1}$. In the free group case $G_1/G_2$, $G_2/G_3$ and $G_3/G_4$ are known \cite{Pettet05,Satoh06} but as yet there is no general formula. In this paper we restrict ourselves to studying the abelianisation of $\torg$, using the generating set $\mathcal{M}_\G$ defined in Section \ref{s:torelli intro}.

\begin{theorem}
The first Johnson homomorphism $\tau_1$ maps $\mathcal{M}_\G$ to a free generating set of a subgroup of $\hom(L_1,L_2)$. The abelianisation of $\torg$ is isomorphic to the free abelian group on the set $\mathcal{M}_\G$, and $G_2$ is the commutator subgroup of $\torg$. 
\end{theorem}

\begin{proof}
By Proposition \ref{l2basis}, the free abelian group $L_2$ has a basis consisting of the set $\{[v_i,v_j].\g_3:i < j, [v_i,v_j]\neq 0\}.$ This allows us to obtain an explicit description of the images of elements of $\mathcal{M}_\G$:
\begin{align*} \tau_1(K_{ij})(v_l)&= \begin{cases} 1.\g_3 & \text{if $v_l \not \in \G_{ij}$} \\
               [v_j,v_l].\g_3 & \text{if $v_l \in \G_{ij}$}
              \end{cases} \\ \tau_1(K_{ijk})(v_l)&= \begin{cases} 1.\g_3 & \text{if $l \neq i$} \\
                [v_j,v_k].\g_3 & \text{if $l = i$}
               \end{cases} \end{align*}
These elements are linearly independent in $\hom(L_1,L_2)$. The second statement follows immediately, and the third follows as $G_2=\ker(\tau_1)$.
\end{proof}

\begin{cor}
 $\mathcal{M}_{\G}$ is a minimal generating set of $\torg$.
                                 \end{cor}

\subsection{Example: the pentagon} Suppose that $\G$ is the pentagon shown in Figure~1.
\begin{figure} [ht]
\includegraphics{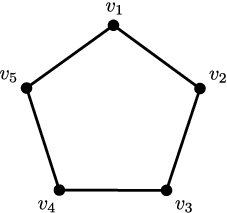}
\centering
\caption{The pentagon graph.}
\label{fig:pentagon}
\end{figure}
In this case, if $v_i \leq v_j$ then $v_i = v_j$, therefore no elements of the form $K_{ijk}$ exist in $\torg$. Removing $st(v_i)$ from $\G$ leaves exactly one connected component consisting of the two vertices opposite $v_i$, therefore
$$\mathcal{M}_\G=\{K_{13},K_{24},K_{35},K_{41},K_{52}\}.$$
Hence $H_1(\torg)=G_1/G_2 \cong \mathbb{Z}^5$. Also, $\{[v_1,v_3],[v_1,v_4],[v_2,v_4],[v_2,v_5],[v_3,v_5]\}$ is a set of coset representatives of $\g_2$ in $\g_3$, therefore $\hom(L_1,L_2) \cong \mathbb{Z}^{25}.$ In particular $\tau_1$ is not surjective (in contrast to the free group situation --- see \cite{Pettet05}).

\subsection{Extension to $\outg$}
We'd now like to move our attention to the images of $G_1,G_2,\ldots$ in $\outg$, which we will label $H_1,H_2, \ldots$ respectively. Let $\pi$ be the natural projection $\autg \to \outg$.   The action of an element of $A_\G$ on $A_\G$ by conjugation induces a homomorphism $ad\co A_\G \to \autg$.

\begin{lemma}\label{l:e1}
If $Z(A_\G)=\{1\}$, then $g \in \g_c$ if and only if $ad(g) \in G_c$.
\end{lemma}

\begin{proof}
If $g \in \g_c$ then $ghg^{-1}.\g_{c+1}=h.\g_{c+1}$ for all $h \in A_\G$. Hence $ad(g) \in G_c$. Conversely, suppose that $g \not \in \g_c$. Then there exists a unique integer $d <c$ with $g \in \g_d$ and $g \not \in \g_{d+1}$. By Proposition~\ref{pp:centres}, there exists $v \in V(\G)$ such that $[g,v] \not \in \g_{d+2} $. As $\g_{c+1} \subset \g_{d+2}$ we have $gvg^{-1}.\g_{c+1} \neq v. \g_{c+1}$. Hence $ad(g) \not \in G_c$. 
\end{proof}

\begin{proposition} \label{p:e}
When $Z(A_\G)=\{1\}$ we have an exact sequence of abelian groups:
$$ 0 \to \g_{c}/\g_{c+1} \xrightarrow{\alpha} G_c/G_{c+1} \xrightarrow{\beta} H_c/H_{c+1} \to 0,$$
where $\alpha$ and $\beta$ are induced by $ad$ and $\pi$ respectively. \end{proposition}

\begin{proof}
By Lemma~\ref{l:e1}, the map $\alpha$ is injective. The sequence $H_1, H_2, \ldots,$ is defined to be the image of $G_1, G_2, \ldots $ in $\outg$, so the map $\beta$ is surjective. Furthermore, $\pi(ad(A_\G))=\{1\}$, so that $\im(\alpha) \subset \ker(\beta)$. It remains to check that $\ker(\beta)\subset \im(\alpha)$.  Suppose that $\phi \in G_c$ and $[\phi] \in H_{c+1}$. Then there exists $g \in A_\G$ and $\psi \in G_{c+1}$ such that $\phi=ad(g)\psi$. Then $ad(g)=\phi\psi^{-1} \in G_c$, so by Lemma~\ref{l:e1} we have $g \in \g_c$. Hence $\phi. G_{c+1}=ad(g). G_{c+1}$ and $\phi .G_{c+1}$ is in the image of $\g_c/\g_{c+1}$.\end{proof}

\begin{theorem}
If $Z(A_\G)=\{1\}$ and $c \geq 1$ then the group $H_c /H_{c+1}$ is torsion-free.
\end{theorem}

\begin{proof}
Let $[\phi] \in H_c$ with $\phi \in G_c$ and suppose that $[\phi]^p.H_{c+1}=0$. Then $\phi^p.G_{c+1}$ lies in $\ker \beta$ and by the exact sequence in Proposition~\ref{p:e} there exists $g \in \g_c$ such that $\phi^p.G_{c+1}=ad(g).G_{c+1}$. As $\phi \in G_c,$ for every $v \in V(\G)$ there exists $w \in \g_{c+1}$ such that $\phi(v)=vw$. Hence $\tau_c(\phi)(v)=\phi(v)v^{-1}.\g_{c+2}=w.\g_{c+2}$. Therefore \begin{align*} [g,v].\g_{c+2} &= ad(g)(v)v^{-1}.\g_{c+2} \\
&=\tau_c(ad(g))(v) &&\text{(definition of $\tau_c$)}\\
&=\tau_c(\phi^p)(v) &&\text{($ad(g).G_{c+1}=\phi^p.G_{c+1}$)} \\                                                                                                                                                                                                                                                                                                                                              
&=\tau_c(\phi)^p(v) &&\text{($\tau_c$ is a homomorphism)}\\
&=w^p.\g_{c+2}.                                                                                                                                                                                                                                                                                                                                                                                                                                                                 \end{align*}
From Proposition~\ref{pp:centres} it follows that there exists $h \in \g_c$ such that $h^p .\g_{c+1}=g.\g_{c+1}$, and applying $ad$ gives $ad(h)^p.G_{c+1}=ad(g).G_{c+1}=\phi^p.G_{c+1}$. As $G_c/G_{c+1}$ is a free abelian group, it has unique roots, so that $ad(h).G_{c+1}=\phi.G_{c+1}$. Hence $[\phi].H_{c+1}=0$, and $H_c/H_{c+1}$ is torsion-free. 
\end{proof}

We now adapt a well-known fact about $\outn$ to $\outg$.

\begin{proposition}
The intersection $\cap_{c=1}^{\infty}H_c$ is trivial.
\end{proposition}

\begin{proof}
Let $\phi \in \autg$, and suppose that $[\phi] \in H_c$ for all $c$. Then for every element $g \in A_\G$, we know that $\phi(g)$ is conjugate to $g$ in $A_\G/\g_c$ for all $c$. Toinet has shown that RAAGs are conjugacy separable in finite $p$--group, (and therefore nilpotent) quotients \cite{Toinet10};  this tells us that $\phi(g)$ is conjugate to $g$. Furthermore, Minasyan (\cite{Minasyan2009}, Proposition 6.9) has shown that if $\phi$ takes every element of $A_\G$ to a conjugate, then $\phi$ itself is an inner automorphism. Hence $\cap_{c=1}^{\infty}H_c=\{1\}$.
\end{proof}

Therefore if $Z(A_\G)$ is trivial then $H_1,H_2,H_3\ldots$ is central series of $\toro$, with trivial intersection, such that the consecutive quotients $H_c/H_{c+1}$ are free abelian.

\begin{cor} \label{c:tor}
If $Z(A_\G)$ is trivial then $\toro$ is residually torsion-free nilpotent.
\end{cor}

\subsection{The situation when $Z(A_\G)$ is nontrivial}

Suppose that $Z(A_\G)$ is nontrivial. Let $[v]$ be the unique maximal equivalence class of vertices in $\G$. By Proposition~\ref{p:projections2}, there is a restriction map $R_v$ and a projection map $P_v$ like so:
\begin{align*} R_v&\co \outg \to \gl(Z(A_\G)) \cong \gl(A_{[v]}) \\
 P_v&\co \outg \to \out(A_\G/Z(A_\G)) \cong \out(A_{lk[v]}).
\end{align*}
The kernel of the map $R_v\times P_v$ is the free abelian subgroup $\text{Tr}$ generated by the transvections $[\r_{ij}]$ such that $v_i \in lk[v]$ and $v_j \in [v]$. Elements of $\outg$ that lie in $\text{Tr}$ act non-trivially on $H_1(A_\G)$, as do elements that are nontrivial under $R_v$. It follows that $\toro$ is mapped isomorphically under $P_v$ onto $\overline{\mathcal{T}}(A_{lk[v]})$. As the centre of $A_{lk[v]}$ is trivial, this lets us promote the above work to any right-angled Artin group:

\begin{theorem} \label{t:tfn}
For any graph $\G$, the group $\overline{\mathcal{T}}(A_\G)$ is residually torsion-free nilpotent.
\end{theorem}

\section{$\SL$--dimension and projection homomorphisms for subgroups of $\outg$.} \label{s:sldim}

In Section \ref{s:main} we will restrict the actions of certain groups on RAAGs of bounded $\SL$--dimension by using an induction argument based on the number of vertices in $\G$, combined with projection homomorphisms.  Unfortunately our definition of $\SL$--dimension described in the introduction can behave badly under projection, restriction, and inclusion homomorphisms. In particular, if, $v$ is a maximal vertex in a connected graph $\G$, then it is not always true that $d_{SL}(lk[v]) \leq d_{SL}(\G)$. To get round this problem, we define the $\SL$--dimension $\dsl(G)$ of an arbitrary subgroup $G \leq \outg$ and show that if instead we look at the image of $\outg$ under such homomorphisms, then the $\SL$--dimension will not increase (for instance, it will always be the case that $d_{SL}(P_v(\outg)) \leq d_{SL}(\outg)).$

\begin{proposition}
Let $G \leq \outg$ be generated by the subset $T=G\cap\mathcal{S}_\G$ of our standard generating set $\mathcal{S}_\G$ of $\out(A_\G)$. Let $G^0$ be the subgroup of $G$ generated by the extended partial conjugations, inversions, and transvections in $T$ and let ${\rm{PT}}G^0$ be the subgroup of $G$ generated solely by the extended partial conjugations and transvections in $T$. Then $G^0$ and ${\rm{PT}}G^0$ are finite-index normal subgroups of $G$.
\end{proposition}

\begin{proof}
Suppose $\alpha$ is a graph symmetry that moves the vertices according to the permutation $\sigma$. We find that $\alpha K_{ij}\alpha^{-1}=K_{\sigma(i)\sigma(j)}$, $\alpha \r_{ij} \alpha^{-1}=\r_{\sigma(i)\sigma(j)}$ and $\alpha s_i \alpha^{-1}=\alpha_{\sigma(i)}$. As $T=G \cap \mathcal{S}_\G$, if $[\phi]$ and $[\alpha]$ belong to $T$, then so does $[\alpha \phi \alpha^{-1}]$. Therefore if $W$ is a word in $T \cup T^{-1}$ one may shuffle graph symmetries along so that they all occur at the beginning of $W$. As the group of graph symmetries is finite, this shows that $G^0$ is finite index in $G$, and the above computations verify that $G^0$ is normal in $G$. Similarly, with inversions one verifies that: $$ \r_{kl}s_i=\begin{cases}  s_i\r_{kl} & i \neq k,l \\
                                                                                                                                       s_i\r_{ki}^{-1} & i=l \\
s_iK_{il}^{-1}\r_{il}^{-1} & i=k,[v_i,v_l]\neq0 \\
s_i\r_{il}^{-1} & i=k,[v_i,v_l]=0\end{cases} \text{and} \quad  K_{kl}s_i=\begin{cases} s_iK_{kl} & i \neq l \\ s_iK_{kl}^{-1} & i=l \end{cases}.$$
These show that ${\rm{PT}}G^0$ is normal in $G^0$, and we may write any element of $G^0$ in the form $[s_1^{\epsilon_1}\ldots s_n^{\epsilon_n}\phi']$, where $\epsilon_i \in \{0,1\}$ and $\phi'$ is a product of extended partial conjugations and transvections. Therefore ${\rm{PT}}G^0$ is of index at most $2^n$ in $G^0$.
\end{proof}

Now lets look at the generators of $\autg$ (respectively $\outg$) under the map $\Phi: \autg \to \gln$ (respectively $\overline{\Phi}:\outg \to \gln$). If $\alpha$ is a graph symmetry, then $\Phi(\alpha)$ is the appropriate permutation matrix corresponding to the permutation $\alpha$ induces on the vertices. For a partial conjugation $K_{ij}$ we see that $\Phi(K_{ij})$ is the identity matrix $I$; $\Phi$ sends the inversion $s_i$ to the matrix $S_i$ which has $1$ everywhere on the diagonal except for $-1$ at the $(i,i)th$ position, and zeroes everywhere else, and $\Phi$ sends the transvection $\r_{ij}$ to the matrix $T_{ji}=I + E_{ji}$, where $E_{ji}$ is the elementary matrix with $1$ in the $(i,j)$th position, and zeroes everywhere else. The swapping between $\r_{ij}$ and $T_{ji}$ may seem a little unnatural, but occurs as a choice of having $\autg$ act on the left.  It follows that the image of $\saut$ under $\Phi$ is the subgroup of $\gln$ generated by matrices of the form $T_{ij}$, where $v_j \leq v_i$. 
After restricting our attention to $G^0$ and ${\rm{PT}}G^0$ our main tool for studying the image of $\outg$ and its subgroups under $\overline{\Phi}$ will be by ordering the vertices of $\G$ in the manner described in Sections \ref{s:relation} and \ref{s:gordering}. If we order the vertices as in Section \ref{s:relation}, then $[\r_{ij}] \in \outg$ only if either $v_i$ and $v_j$ are in the same equivalence class of vertices, or $i \leq j$. It follows that a matrix in the image of $\overline{\Phi}|_{\outo}$ has a block decomposition of the form:
 $$ M=\begin{pmatrix}M_1 & 0 & \dots & 0 \\
* & M_2 & \dots & 0 \\
\hdotsfor{4} \\
* & * & \dots & M_{r} 
 \end{pmatrix},$$
 
where the $*$ in the $(i,j)$th entry in the block decomposition may be non-zero if $[v_{m_j}]\leq[v_{m_i}]$, but zero otherwise. 

Similarly, given a subgroup $G\leq \outg$ generated by a subset of $\mathcal{S}_\G$, if we order the vertices by the method given in Section~\ref{s:gordering}, a matrix in the image of $\overline{\Phi}|_{G^0}$ has a block decomposition of the form:
$$ M=\begin{pmatrix}N_1 & 0 & \dots & 0 \\
* & N_2 & \dots & 0 \\
\hdotsfor{4} \\
* & * & \dots & N_{r'} 
 \end{pmatrix},$$
 
 where the $*$ in the $(i,j)$th entry in the block decomposition may be non-zero if $[v_{l_j}]_G\leq[v_{l_i}]_G$, but zero otherwise. If $[v]_G$ is an abelian equivalence class of vertices, then $G$ contains a copy $\SL(A_{[v]_G})$ generated by the $[\r_{ij}]$ with $v_i,v_j \in [v]$. This fact lends itself to the following definition:
 
\begin{definition}
For any (not necessarily finitely generated) subgroup $G \leq \outg$, the $\SL$--dimension of $G$, written $d_{SL}(G)$, is defined to be the size of the largest abelian equivalence class under $\sim_G$. If there are no abelian equivalence classes under $\sim_G$, define $d_{SL}(G)=1$.
\end{definition}

\begin{rem}
A result of Droms \cite{MR891135} tells us that $\G$ is determined by $A_\G$, however $\G$ is not determined by $\outg$. Hence the definition of $\dsl(G)$ depends on our choice of $\G$ and the embedding $G \leq \outg$, as the definition of the equivalence relation $\sim_G$ relies on our chosen generating set $\mathcal{S}_\G$. This should not cause confusion in the work that follows -- we will only study subgroups $G \leq \outg$ generated by subsets of $\mathcal{S}_\G$.
\end{rem}

Roughly speaking, $\dsl(G)$ is the largest integer such that $G$ contains an obvious copy of $SL_{d_{SL}(G)}(\mathbb{Z})$. Note that $d_{SL}(\outg)$ is simply the size of the largest abelian equivalence class under the relation $\leq$ defined on the vertices, so is equal to $\dsl(\G)$. As each abelian equivalence class of vertices is a clique in $\G$, the $\SL$--dimension of $\outg$ is less than or equal to the size of a maximal clique in $\G$ (this is known as the \emph{dimension of} $A_\G$). We can now look at how $G$ and its $\SL$--dimension behave under restriction, exclusion, and projection maps.

\begin{lemma}\label{l:dimr}
Let $G \leq \outg$ be generated by $T=G\cap\mathcal{S}_\G$. Suppose that $\G'$ is a full subgraph of $\G$ and the conjugacy class of  $A_{\G'}$ in $A_\G$ is preserved by $G$. Then under the restriction map $R_{\G'}$,  the group $R_{\G'}(G)\leq \out(A_{\G'})$ is generated by a subset of $\mathcal{S}_{\G'}$, and $\dsl(R_{\G'}(G))\leq \dsl(G)$.
\end{lemma}

\begin{proof}
One first checks that for an element $[\phi] \in T$, either $R_{\G'}([\phi])$ is trivial or $R_{\G'}([\phi]) \in \mathcal{S}_{\G'}$. This is obvious in the case of graph symmetries, inversions, and transvections. In the case of partial conjugations if $v_j$ is not in $\G'$, or if $\G_{ij}\cap\G'=\emptyset$, then $R_{\G'}([K_{ij}])$ is trivial. Otherwise, $\G_{ij} \cap \G'$ is a union of connected components of $\G' - st(v_j)$, so that $R_{\G'}([K_{ij}])$ is an extended partial conjugation of $A_{\G'}$. This proves the first part of the lemma. To prove the second part of the lemma, we first give an alternate definition of $\dsl(G)$ when $G$ is generated by $T=G\cap\mathcal{S}_\G$. Elements in the image of $G^0$ under $\overline{\Phi}$ are of the form:
\begin{equation} \label{matrix} M=\begin{pmatrix}N_1 & 0 & \dots & 0 \\
* & N_2 & \dots & 0 \\
\hdotsfor{4} \\
* & * & \dots & N_{r'} 
 \end{pmatrix}, \end{equation} 
where each $N_i$ is an invertible matrix of size $l_{i+1}-l_i.$ Each of the blocks is associated to either an abelian or non-abelian equivalence class in $\sim_G$, so $\dsl(G)$ is the size of the largest diagonal block in this decomposition associated to an abelian equivalence class (or equal to 1 if there are no abelian equivalence classes). The action of $R_{\G'}(G)^0$ on $H_1(A_{\G'})$ is obtained by removing rows and columns from the decomposition given in Equation \eqref{matrix}, and this gives the block decomposition associated to $\sim_{R_{\G'}(G)}$. It follows that the equivalence classes of $\sim_{R_{\G'}(G)}$ are subsets of the equivalence classes of $\sim_G$. Any abelian equivalence class in $\sim_{R_{\G'}(G)}$ containing at least two vertices will then be a subset of an abelian equivalence class in $\sim_G$. Hence $\dsl(R_{\G'}(G))\leq \dsl(G)$.
\end{proof} 

The following lemma is shown in the same way:

\begin{lemma}\label{l:dime}
Let $G \leq \outg$ be generated by $T=G\cap\mathcal{S}_\G$. Let $\G'$ be a full subgraph of $\G$. Suppose that the normal subgroup  generated by $A_{\G'}$ in $A_\G$ is preserved by $G$. Then under the exclusion map $E_{\G'}$,  the group $E_{\G'}(G) \leq \out(A_{\G-\G'})$ is generated by a subset of $\mathcal{S}_{\G-\G'}$, and $\dsl(E_{\G'}(G))\leq \dsl(G)$.
\end{lemma}

As projection maps are obtained by the concatenation of a restriction and an exclusion map, combining the previous two lemmas gives the following:

\begin{proposition}\label{p:dimp}
Let $G \leq \outg$ be generated by $T=G\cap\mathcal{S}_\G$. Suppose that $\G$ is connected and $v$ is a maximal vertex of $\G$. Under the projection homomorphism $P_v$ of Example \ref{e:maps}, the group $P_v(G^0)\leq \out(A_{lk[v]})$ is generated by a subset of $\mathcal{S}_{lk[v]}$ and $\dsl(P_v(G^0))\leq \dsl(G)=\dsl(G^0)$.
\end{proposition}

\section{Proof of Theorem \ref{t:main}}\label{s:main}

Suppose that a group $\Lambda$ admits no surjective homomorphisms to $\mathbb{Z}$, so that $\hom(\Lambda,\mathbb{Z})=0$. If a group $H$ satisfies the property: 
                                                                                                                                                                   \begin{quote}$(\ast)$ \emph{For every nontrivial subgroup $H'\leq H$ there exists a surjective homomorphism $H'\to\mathbb{Z}$.}\end{quote}

Then $\hom(\Lambda,H)=0$, also. For instance, a simple induction argument on nilpotency class gives the following:

\begin{proposition}\label{p:nilpotent}
Suppose that $\hom(\Lambda,\mathbb{Z})=0$ and $H$ is a finitely generated torsion-free nilpotent group. Then every homomorphism $f:\Lambda \to H$ is trivial.
\end{proposition}

By Theorem \ref{t:tfn} we know that $\overline{\mathcal{T}}(A_\G)$ is residually torsion-free nilpotent. Combining this with Proposition \ref{p:nilpotent} gives:

\begin{proposition}\label{p:tor}
Suppose that $\hom(\Lambda,\mathbb{Z})=0$ and $f:\Lambda \to \toro$ is a homomorphism. Then $f$ is trivial.
\end{proposition}

\begin{proof}
Suppose that $f:\Lambda \to \toro$ is nontrivial. Pick $x \in \Lambda$ with $f(x) \neq 1$. As $\toro$ is residually torsion-free nilpotent there exists a homomorphism $g:\toro \to H$ onto a torsion-free nilpotent group $H$ with $g(f(x)) \neq 1$. As $\toro$ is finitely generated, so is $H$. Then $gf:\Lambda \to H$ is nontrivial, contradicting Proposition~\ref{p:nilpotent}.  
\end{proof}

The overriding theme here is that we may build homomorphism rigidity results from weaker criteria by carefully studying a group's subgroups and quotients. This is also the flavour of our main theorem: 

\begin{theorem}\label{t:main} Suppose that $G$ is a subgroup of $\outg$ generated by a subset $T \subset \mathcal{S}_\G$ and $\dsl(G) \leq m$. Let $$F(\G')=\max\{|V(\G')|:\G' \subset \G \text{ and $A_{\G'}\leq A_\G$ is a free group}\}.$$ Let $\Lambda$ be a group. Suppose that for each finite index subgroup $\Lambda'\leq\Lambda$, we have: \begin{itemize}
                                                                                                                                           \item Every homomorphism $\Lambda' \to \slm$ has finite image, 
\item For all $N \leq F(\G)$, every homomorphism $\Lambda' \to \out(F_N)$ has finite image.
\item $\hom(\Lambda',\mathbb{Z})=0$.                                                                                                                                   \end{itemize}
Then every homomorphism $f:\Lambda \to G$ has finite image.\end{theorem}

The remainder of this section will be dedicated to a proof of Theorem \ref{t:main}. We proceed by induction on the number of vertices in $\Gamma$.  If $\G$ contains only one vertex, then $\outg \cong \mathbb{Z}/2\mathbb{Z}$, so there is no work to do. As the conditions on $\Lambda$ are also satisfied by finite index subgroups, we shall allow ourselves to pass to such subgroups without further comment.

\begin{rem}If either $m \geq 2$ or $F(\G) \geq 2$, then as there exist no homomorphisms from $\Lambda'$ to $\slm$ or $\out(F_{F(\G)})$ with infinite image, it follows that $\hom(\Lambda',\mathbb{Z})=0$ also. This is always the case when $G=\outg$ and $|V(\G)| \geq 2$. Hence the above statement of Theorem \ref{t:main} is a strengthening of the version given in the introduction. \end{rem} 

Let $f:\Lambda \to G$ be such a homomorphism. There are three cases to consider: $\G$ is disconnected, $\G$ is connected and $Z(A_\G)$ is trivial, or $\G$ is connected and $Z(A_\G)$ is nontrivial.

\subsection{$\Gamma$ is disconnected.}
In this case $A_\G \cong F_N \ast_{i=1}^k A_{\G_i}$, where each $\G_i$ is a connected graph containing at least two vertices. Let $\Lambda'=f^{-1}(\outo).$ As $\outo$ is finite index in $\outg$, this means $\Lambda'$ is finite index in $\Lambda$. As in Example \ref{e:disconnected}, for each $\G_i$ there is a restriction homomorphism:$$R_i:\outo \to \out^0(A_{\G_i}).$$
By Lemma \ref{l:dimr}, $R_i(G)$ is generated by a subset $T_i \subset \mathcal{S}_{\G_i}$, and $\dsl(R_i(G)) \leq \dsl(G)$. As $\G_i$ is a proper subgraph of $\G$, we have $F(\G_i)\leq F(\G)$ and $|V(\G_i)| < |V(\G)|$. Hence, by induction $R_if(\Lambda')$ is finite for each $i$, and there exists a finite index subgroup $\Lambda_i$ of $\Lambda'$ such that $R_if(\Lambda_i)$ is trivial. We may also consider the exclusion homomorphism: $$E:\outg \to \out(F_N).$$ As $N \leq F(\G)$, the group $\ker(Ef)$ is a finite index subgroup of $\Lambda$. Let $$\Lambda''=\cap_{i=1}^k\Lambda_i\cap\ker(Ef).$$
As $\Lambda''$ is the intersection of a finite number of finite index subgroups, it is also finite index in $\Lambda$. We now study the action of $\Lambda''$ on $H_1(A_\G)$. The transvection $[\r_{ij}]$ belongs to $\outo$ only if either $v_i$ and $v_j$ belong to the same connected component of $\G$, or if $v_i$ is an isolated vertex of $\G$. Therefore the action of $\outo$ on $H_1(A_\G)$ has a block decomposition of the following form:
$$\begin{pmatrix}M_1 & 0 & \dots & 0 & * \\
0 & M_2 & \dots & 0 & * \\
\hdotsfor{5} \\
0 & 0 & \dots & M_k & * \\
0 & 0 & \dots & 0 & M_{k+1} 
 \end{pmatrix},$$
where $M_i$ corresponds to the action on $A_{\G_i}$, and $M_{k+1}$ corresponds to the action on $F_N$. Moreover, as $R_if(\Lambda'')$ is trivial for each $i$, and $Ef(\Lambda'')$ is trivial, the action of $\Lambda''$ on $H_1(A_\G)$ is of the form:
$$\begin{pmatrix}I & 0 & \dots & 0 & * \\
0 & I & \dots & 0 & * \\
\hdotsfor{5} \\
0 & 0 & \dots & I & * \\
0 & 0 & \dots & 0 & I 
 \end{pmatrix}.$$
This means there is a homomorphism from $\Lambda''$ to an abelian subgroup of $\gln$. As $\hom(\Lambda'',\mathbb{Z})=0$, this homomorphism must be trivial. Hence $f(\Lambda'') \subset \toro.$ By Proposition \ref{p:tor}, this shows that $f(\Lambda'')$ is trivial. Hence $f(\Lambda)$ is finite.
\subsection{$\G$ is connected and $Z(A_\G)$ is trivial.}
Let $\Lambda'=f^{-1}(\outo)=f^{-1}(G^0)$. For each maximal vertex $v$ of $\G$ we have a projection homomorphism:$$P_v:\outo \to \out^0(A_{lk[v]}).$$  By Proposition~\ref{p:dimp}, $P_v(G^0)$ is generated by a subset $T_v \subset \mathcal{S}_{lk[v]}$ and $\dsl(P_v(G^0))\leq \dsl(G^0)=\dsl(G)$. As $lk[v]$ is a proper subgraph of $\G$, we have $F(lk[v])\leq F(\G)$ and $|V(lk[v])| < |V(\G)|$. Therefore by induction $P_vf(\Lambda')$ is finite. Let $$\Lambda''=\bigcap_{[v] \text{ max.}}\ker (P_vf).$$
Then $\Lambda''$ is a finite index subgroup of $\Lambda$ and lies in the kernel of the amalgamated projection homomorphism: $$P:\outo \to \bigoplus_{[v] \text{ max.}}\out^0(A_{lk[v]}).$$  By Theorem \ref{t:projections}, $\ker P$ is a finitely generated free-abelian group. As $\hom(\Lambda'',\mathbb{Z})=0$, a homomorphism from $\Lambda''$ to $\ker P$ must be trivial. Therefore $f(\Lambda'')$ is trivial and $f(\Lambda)$ is finite.

\subsection{$\G$ is connected and $Z(A_\G)$ is nontrivial.}

Suppose that $Z(A_\G)$ is nontrivial. Let $[v]$ be the unique maximal equivalence class in $\G$, so that $Z(A_\G)=A_{[v]}$. Let $P_v$ and $R_v$ be the restriction and projection maps given in Proposition~\ref{p:projections2} so that:
\begin{align*} R_v\co & \outg \to \out(A_{[v]})\cong \gl(A_{[v]}) \\ P_v\co& \outg \to \out(A_{lk[v]}) \end{align*}
If $[v]$ is not equal to the whole of $\G$ then by induction $P_vf(\Lambda)$ and $R_vf(\Lambda)$ are both finite, and there exists a finite index subgroup $\Lambda'$ of $\Lambda$ such that $f(\Lambda')$ is contained in the kernel of $P_v \times R_v$. By Proposition~\ref{p:projections2} this is the free abelian group $\text{Tr}$, so the image of $\Lambda'$ in $\text{Tr}$ is trivial, and $f(\Lambda)$ is finite.

We may then assume that $\G=[v]$. We now look at the $\sim_G$ equivalence classes in $\G$. As $A_\G$ is free abelian, each $[v_i]_G \subset [v]$ is abelian, and as $\dsl(G) \leq m$, every such $[v_i]_G$ contains at most $m$ vertices.  Therefore matrices in (the image of) $G^0$ (under $\overline{\Phi}$) are of the form:

$$ M=\begin{pmatrix}N_1 & 0 & \dots & 0 \\
* & N_2 & \dots & 0 \\
\hdotsfor{4} \\
* & * & \dots & N_{r'} 
 \end{pmatrix},$$

where the $*$ in the $(i,j)$th block is possibly non-zero if $[v_{l_j}]\leq[v_{l_i}]$. For each $i$, we can look at the projection $M \mapsto N_{i}$ to obtain a homomorphism $g_i:{\rm{PT}}G^0 \to \text{SL}_{l_{i+1}-l_i}(\mathbb{Z})$. As $l_{i+1}-l_i \leq m$, our hypothesis on $\Lambda$ implies that $g_if(f^{-1}({\rm{PT}}G^0))$ is finite for all $i$. Let $\Lambda_i$ be the kernel of each map $g_if$ restricted to $f^{-1}({\rm{PT}}G^0)$. Each $\Lambda_i$ is finite index in $\Lambda$. Let $\Lambda'= \cap_{i=1}^k \Lambda_i.$ Then matrices in the image of $\Lambda'$ under $f$ are of the form:
 $$ M=\begin{pmatrix}I & 0 & \dots & 0 \\
* & I & \dots & 0 \\
\hdotsfor{4} \\
* & * & \dots & I 
 \end{pmatrix},$$
and $f(\Lambda')$ is a torsion-free nilpotent group. By Proposition \ref{p:nilpotent} this implies that $f(\Lambda')$ is trivial, so $f(\Lambda)$ is finite and this finishes the final case of the theorem.

\section{Consequences of Theorem \ref{t:main}} \label{s:consequences}

We say that a group $\Lambda$ is $\mathbb{Z}$--\emph{averse} if no finite index subgroup of $\Lambda$ contains a normal subgroup that maps surjectively to $\mathbb{Z}$.  This restriction gives a large class of groups to which $\Lambda$ cannot map. For instance, in \cite{BW2010} Bridson and the author prove the following theorem:

\begin{theorem}\label{t:free}
Suppose that $\Lambda$ is a finitely generated $\mathbb{Z}$--averse group and $f:\Lambda \to \outn$ is a homomorphism. Then $f(\Lambda)$ is finite.
\end{theorem}

In \cite{BW2010} the theorem is incorrectly stated as being true for infinitely generated groups. The proof uses the fact that every finitely generated fully irreducible subgroup of $\out(F_n)$ contains a fully irreducible element \cite{HM}. Finite generation was missing from the statement of this result in \cite{HM}, but corrected in \cite{HM2}.

Note that if $\Lambda$ is $\mathbb{Z}$--averse, then every finite index subgroup of $\Lambda$ is also $\mathbb{Z}$--averse. As there are no homomorphisms from a $\mathbb{Z}$--averse group to $\SL_2(\mathbb{Z})$ with infinite image (as $\SL_2(\mathbb{Z})$ is virtually free), combining the above with Theorem \ref{t:main} we obtain:

\begin{cor}
If $\Lambda$ is a finitely generated $\mathbb{Z}$--averse group, and $\G$ is a finite graph that satisfies ${\dsl(\G)\leq 2},$ then every homomorphism $f:\Lambda \to \outg$ has finite image.
\end{cor}

We would like to apply Theorem \ref{t:main} to irreducible lattices in higher-rank Lie groups. For the remainder of this section $\Lambda$ will be an irreducible lattice in a semisimple real Lie group $G$ with real rank $\rrank G \geq 2$, finite centre, and no compact factors. Such lattices are $\mathbb{Z}$-averse by Margulis' normal subgroup theorem, which states that if $\Lambda'$ is a normal subgroup of $\Lambda$ then either $\Lambda/\Lambda'$ is finite, or $\Lambda'\subset Z(G)$, so $\Lambda'$ is finite. The work of Margulis also lets us restrict the linear representations of such lattices:

\begin{proposition} \label{p:SLrigidity}
If $\rrank G \geq k$ then every homomorphism $f:\Lambda \to \SL_k(\mathbb{Z})$ has finite image. \end{proposition}

To prove this we appeal to Margulis superrigidity. The following two theorems follow from \cite{Margulis91}, Chapter IX, Theorems 6.15 and 6.16 and the remarks in 6.17:

\begin{theorem}\label{t:ss}
Let H be a real algebraic group and $f:\Lambda \to H$ a homomorphism. The Zariski closure of the image of $f$, denoted $\overline{f(\Lambda)}$, is semisimple.
\end{theorem}

\begin{theorem}[Margulis' Superrigidity Theorem]\label{t:sr}
Let $H$ be a connected, semisimple, real algebraic group and $f:\Lambda \to H$ a homomorphism. If
\begin{itemize}
\item H is adjoint (equivalently $Z(H)=1$) and has no compact factors, and
\item $f(\Lambda)$ is Zariski dense in H,
\end{itemize}
then $f$ extends uniquely to a continuous homomorphism $\tilde{f}:G \to H$. Furthermore, if $Z(G)=1$ and $f(\Lambda)$ is nontrivial and discrete, then $\tilde{f}$ is an isomorphism.
\end{theorem}

We may combine these to prove Proposition \ref{p:SLrigidity}:

\begin{proof}[Proof of Proposition \ref{p:SLrigidity}]
Let $f:\Lambda \to \SL_k(\mathbb{Z})$ be a homomorphism. By Theorem~\ref{t:ss}, the Zariski closure of the image $\overline{f(\Lambda)} \subset \SL_k(\mathbb{R})$ is semisimple. Also, $\overline{f(\Lambda)}$ has finitely many connected components --- let $\overline{f(\Lambda)}_0$ be the connected component containing the identity. Decompose $\overline{f(\Lambda)}_0=H_1 \times K$, where $K$ is a maximal compact factor. Then $H_1$ is a connected semisimple real algebraic group with no compact factors. We look at the finite index subgroup $\Lambda_1=f^{-1}(H_1)$ of $\Lambda$, so that $\overline{f(\Lambda_1)}=H_1$. As the centre of a subgroup of an algebraic group is contained in the centre of its Zariski closure, $f(Z(\Lambda_1)) \subset Z(H_1)$. This allows us to factor out centres in the groups involved. Let $G_2=G/Z(G)$, $\Lambda_2=\Lambda_1/Z(\Lambda_1)=\Lambda_1/(\Lambda_1 \cap Z(G))$ and $H_2=H_1/Z(H_1)$. Then there is an induced map $f_2:\Lambda_2 \to H_2$ satisfying the conditions of Theorem \ref{t:sr}. 
Therefore if $f_2(\Lambda_2)\neq1$ there is an isomorphism $\tilde{f}_2:G_2 \to H_2$. However \begin{align*}\rrank G_2=\rrank G&\geq k \\ \rrank H_2 = \rrank H_1 \leq \rrank \SL_k(\mathbb{R}) &= k-1.\end{align*}
This contradicts the isomorphism between $H_2$ and $G_2$. Therefore $f_2(\Lambda_2)=1$. As $Z(\Lambda_1)$ is finite, and $\Lambda_1$ is finite index in $\Lambda$, this show that the image of $\Lambda$ under $f$ is finite.
\end{proof}

Combining Proposition \ref{p:SLrigidity} with Theorems \ref{t:free} and \ref{t:main}, this gives:

\begin{theorem}\label{t:lr}
Let $G$ be a real semisimple Lie group with finite centre, no compact factors, and $\rrank G \geq 2$. Let $\Lambda$ be an irreducible lattice in $G$. If $\rrank G \geq d_{SL}(\G)$, then every homomorphism $f:\Lambda \to \outg$ has finite image.
\end{theorem}

The following corollary justifies our definition of $\SL$--dimension, and shows that you can't hide any larger copies of $\SL_n(\mathbb{Z})$ inside $\outg$:

\begin{cor} \label{c:dim}
For $k \geq 3$, the group $\outg$ contains a subgroup isomorphic to $\SL_k(\mathbb{Z})$ if and only if $k \leq \dsl(\G)$.
\end{cor}

We can't expect to have such a nice description of when $\out(A_\G)$ contains a copy of $\SL_2(\mathbb{Z})$. As $\SL_2(\mathbb{Z})$ is virtually free, it is easier to embed in other groups than its higher-rank cousins.

\begin{qu}
What properties does $\G$ require in order for $\SL_2(\mathbb{Z}) \leq \out(A_\G)?$ 
\end{qu}

Having $d_{SL}(\G) \geq 2$ is a sufficient condition, but as $\out(F_2) \cong \gl_2(\mathbb{Z})$ it is certainly not necessary.

\bibliography{RAAGreferences}{}

\begin{thebibliography}{10}

\bibitem{BL}
Hyman Bass and Alexander Lubotzky.
\newblock Linear-central filtrations on groups.
\newblock In {\em The mathematical legacy of {W}ilhelm {M}agnus: groups,
  geometry and special functions ({B}rooklyn, {NY}, 1992)}, volume 169 of {\em
  Contemp. Math.}, pages 45--98. Amer. Math. Soc., Providence, RI, 1994.

\bibitem{BB}
Mladen Bestvina and Noel Brady.
\newblock Morse theory and finiteness properties of groups.
\newblock {\em Invent. Math.}, 129(3):445--470, 1997.

\bibitem{BF}
Mladen Bestvina and Mark Feighn.
\newblock A hyperbolic {${\rm Out}(F_n)$}-complex.
\newblock {\em Groups Geom. Dyn.}, 4(1):31--58, 2010.

\bibitem{MR979493}
Nicolas Bourbaki.
\newblock {\em Lie groups and {L}ie algebras. {C}hapters 1--3}.
\newblock Elements of Mathematics (Berlin). Springer-Verlag, Berlin, 1998.

\bibitem{bridson-farb}
Martin~R. Bridson and Benson Farb.
\newblock A remark about actions of lattices on free groups.
\newblock {\em Topology Appl.}, 110(1):21--24, 2001.
\newblock Geometric topology and geometric group theory (Milwaukee, WI, 1997).

\bibitem{BW2010}
Martin~R. Bridson and Richard~D. Wade.
\newblock Actions of higher-rank lattices on free groups.
\newblock {\em Compos. Math.}, 147(5):1573--1580, 2011.

\bibitem{MR2322545}
Ruth Charney.
\newblock An introduction to right-angled {A}rtin groups.
\newblock {\em Geom. Dedicata}, 125:141--158, 2007.

\bibitem{MR2372847}
Ruth Charney, John Crisp, and Karen Vogtmann.
\newblock Automorphisms of 2-dimensional right-angled {A}rtin groups.
\newblock {\em Geom. Topol.}, 11:2227--2264, 2007.

\bibitem{CV2009}
Ruth Charney and Karen Vogtmann.
\newblock Finiteness properties of automorphism groups of right-angled {A}rtin
  groups.
\newblock {\em Bull. Lond. Math. Soc.}, 41(1):94--102, 2009.

\bibitem{CV2010}
Ruth Charney and Karen Vogtmann.
\newblock Subgroups and quotients of automorphism groups of {RAAG}s.
\newblock In {\em Low-dimensional and symplectic topology}, volume~82 of {\em
  Proc. Sympos. Pure Math.}, pages 9--27. Amer. Math. Soc., Providence, RI,
  2011.

\bibitem{DGO}
F.~Dahmani, V.~Guirardel, and D.~Osin.
\newblock Hyperbolically embedded subgroups and rotating families in groups
  acting on hyperbolic spaces.
\newblock {\em arXiv:1111.7048}, 11 2011.

\bibitem{Day09}
Matthew~B. Day.
\newblock Symplectic structures on right-angled {A}rtin groups: between the
  mapping class group and the symplectic group.
\newblock {\em Geom. Topol.}, 13(2):857--899, 2009.

\bibitem{Droms2}
Carl Droms.
\newblock {\em Graph groups}.
\newblock PhD thesis, Syracuse University, 1983.

\bibitem{MR891135}
Carl Droms.
\newblock Isomorphisms of graph groups.
\newblock {\em Proc. Amer. Math. Soc.}, 100(3):407--408, 1987.

\bibitem{MR1176154}
G.~Duchamp and D.~Krob.
\newblock The free partially commutative {L}ie algebra: bases and ranks.
\newblock {\em Adv. Math.}, 95(1):92--126, 1992.

\bibitem{DK1}
G.~Duchamp and D.~Krob.
\newblock The lower central series of the free partially commutative group.
\newblock {\em Semigroup Forum}, 45(3):385--394, 1992.

\bibitem{FM}
Benson Farb and Howard Masur.
\newblock Superrigidity and mapping class groups.
\newblock {\em Topology}, 37(6):1169--1176, 1998.

\bibitem{HW}
Fr{{\'e}}d{{\'e}}ric Haglund and Daniel~T. Wise.
\newblock Special cube complexes.
\newblock {\em Geom. Funct. Anal.}, 17(5):1551--1620, 2008.

\bibitem{HM2}
M.~Handel and L.~Mosher.
\newblock Subgroup decomposition in $out(f_n)$: Introduction and research
  announcement.
\newblock {\em arXiv:1302.2681}, 2013.

\bibitem{HM}
Michael Handel and Lee Mosher.
\newblock {S}ubgroup classification in {$Out(F_n)$}.
\newblock {\em arXiv:0908.1255}, 2009.

\bibitem{KMNR}
Ki~Hang Kim, L.~Makar-Limanov, Joseph Neggers, and Fred~W. Roush.
\newblock Graph algebras.
\newblock {\em J. Algebra}, 64(1):46--51, 1980.

\bibitem{L1}
Pierre Lalonde.
\newblock Bases de {L}yndon des alg{\`e}bres de {L}ie libres partiellement
  commutatives.
\newblock {\em Theoret. Comput. Sci.}, 117(1-2):217--226, 1993.

\bibitem{MR1356145}
Michael~R. Laurence.
\newblock A generating set for the automorphism group of a graph group.
\newblock {\em J. London Math. Soc. (2)}, 52(2):318--334, 1995.

\bibitem{MKS}
Wilhelm Magnus, Abraham Karrass, and Donald Solitar.
\newblock {\em Combinatorial group theory}.
\newblock Dover Publications Inc., New York, revised edition, 1976.
\newblock Presentations of groups in terms of generators and relations.

\bibitem{Margulis91}
G.~A. Margulis.
\newblock {\em Discrete subgroups of semisimple {L}ie groups}, volume~17 of
  {\em Ergebnisse der Mathematik und ihrer Grenzgebiete (3) [Results in
  Mathematics and Related Areas (3)]}.
\newblock Springer-Verlag, Berlin, 1991.

\bibitem{Minasyan2009}
Ashot Minasyan.
\newblock Hereditary conjugacy separability of right angled {A}rtin groups and
  its applications.
\newblock {\em Groups. Geom. Dyn.}, 6(2):335--388, 2012.

\bibitem{W-M}
Dave~Witte Morris.
\newblock Introduction to arithmetic groups.
\newblock {\em arXiv:math/0106063}, 2012.

\bibitem{Pettet05}
Alexandra Pettet.
\newblock The {J}ohnson homomorphism and the second cohomology of {${\rm
  IA}_n$}.
\newblock {\em Algebr. Geom. Topol.}, 5:725--740, 2005.

\bibitem{Satoh06}
Takao Satoh.
\newblock New obstructions for the surjectivity of the {J}ohnson homomorphism
  of the automorphism group of a free group.
\newblock {\em J. London Math. Soc. (2)}, 74(2):341--360, 2006.

\bibitem{MR1023285}
Herman Servatius.
\newblock Automorphisms of graph groups.
\newblock {\em J. Algebra}, 126(1):34--60, 1989.

\bibitem{Toinet10}
Emmanuel Toinet.
\newblock Conjugacy p-separability of right-angled {A}rtin groups and
  applications.
\newblock {\em arXiv:1009.3859}, 2010.

\bibitem{W1}
Richard~D. Wade.
\newblock The lower central series of a right-angled {A}rtin group.
\newblock {\em arXiv:1109.1722}, 2011.

\end{thebibliography}
\bibliographystyle{plain}

\end{document}